\documentclass[11pt]{amsart}
\usepackage{amssymb, latexsym, graphicx, mathrsfs, enumerate}

\numberwithin{equation}{section}
\newtheorem{Thm}[equation]{Theorem}
\newtheorem{Prop}[equation]{Proposition}
\newtheorem{Cor}[equation]{Corollary}
\newtheorem{Lem}[equation]{Lemma}
\newtheorem{Conj}[equation]{Conjecture}

\theoremstyle{definition}

\newtheorem{Rmk}[equation]{Remark}

\newcommand{\Q}{\mathbb{Q}}
\newcommand{\K}{\widehat{K}}

\begin{document}

\title[The average of the smallest prime in a conjugacy class]{The average of the smallest prime in a conjugacy class}

\author[Peter Cho]{Peter J. Cho}
\address{Department of Mathematical Sciences, Ulsan National Institute of Science and Technology \\
Ulsan, Korea}
\email{petercho@unist.ac.kr}
\author[Henry Kim]{Henry H. Kim$^{\star}$}
\thanks{$^{\star}$ partially supported by an NSERC grant.}
\address{Department of
Mathematics, University of Toronto, Toronto, ON M5S 2E4, CANADA \\
and Korea Institute for Advanced Study, Seoul, Korea}
\email{henrykim@math.toronto.edu}


\begin{abstract} Let $C$ be a conjugacy class of $S_n$ and $K$ an $S_n$-field. Let $n_{K,C}$ be the smallest prime which is ramified or whose Frobenius automorphism Frob$_p$ does not belong to $C$. Under some technical conjectures, we compute the average of $n_{K,C}$. 
For $S_3$ and $S_4$-fields, our result is unconditional. For $S_n$-fields, $n=3,4,5$, we give a different proof which depends on the strong Artin conjecture. Let $N_{K,C}$ be the smallest prime for which Frob$_p$ belongs to $C$. For $S_3$-fields, we obtain an unconditional result for the average of $N_{K,C}$ for $C=[(12)]$.
\end{abstract}

\maketitle
\section{Introduction}

For a fundamental discriminant $D$, let $\chi_D(\cdot)= \left( \frac{D}{\cdot}\right)$, and let $N_{D,\pm 1}$ be the smallest prime such that 
$\chi_D(p)=\pm 1$, resp. Let $n_{D,\pm 1}$ be the smallest prime such that $\chi_D(p) \neq \pm 1$. 
We can interpret $N_{D,1}$ ($N_{D,-1}$) as the smallest prime which splits completely (inert, resp.) in a quadratic field $\Bbb Q(\sqrt{D})$.
Under the assumption of the Generalized Riemann Hypothesis (GRH) for $L(s,\chi_D)$, one can show easily that
$N_{D,\pm 1}, n_{D,\pm 1}\ll (\log D)^2$. Erd\"{o}s \cite{Erdos} considered the average of those values over a family:
$\displaystyle\lim_{ X \rightarrow \infty} \frac {\sum_{2 < p \leq X} N_{p, -1}}{\pi(X)} = \sum_{k=1}^\infty \frac{p_k}{2^k}=3.67464...$
where $p$ runs through primes, and $p_k$ is the $k$-th prime. Pollack\footnote{In \cite{MP}, the value $4.98094..$ is misquoted as $4.98085$.}  \cite{P} generalized Erd\"{o}s' result to all fundamental discriminants:
\begin{eqnarray*}
\lim_{ X \rightarrow \infty} \frac{\sum_{|D| \leq X} N_{D,\pm1} }{\sum_{|D| \leq X}1} = \sum_{q} \frac{q^2}{2(q+1)}\prod_{p < q} \frac{p+2}{2(p+1)}= 4.98094\dots.
\end{eqnarray*}
Pollack \cite{P2} also computed the average of the least inert primes over cyclic number fields of prime degree. 

We generalize this problem to the setting of general number fields. We call a number field $K$ of degree $n$, an $S_n$-field if its Galois closure $\K$ over $\Q$ is an $S_n$ Galois extension. Let $C$ be a conjugacy class of $S_n$. For an unramified prime $p$, denote by Frob$_p$, a Frobenius automorphism of $p$. Define $n_{K,C}$ to be the smallest prime $p$ which is ramified in $K$ or for which Frob$_p \not\in C$, and define $N_{K,C}$ to be the smallest prime $p$ such that Frob$_p\in C$. Under GRH, we can show that 
$n_{K,C}, N_{K,C}\leq (\log |d_K|)^2$ (cf. \cite{BaS}).

In this paper we consider the average value of 
$n_{K,C}$ over fields in $L_n^{(r_2)}(X)$, which is the set of 
$S_n$-fields $K$ of signature $(r_1,r_2)$ with $|d_K| \leq X$, where $d_K$ is the discriminant of $K$:

\begin{Thm} \label{main}
Let $n=3,4,5$. When $n=5$, we assume either the strong Artin conjecture, or Conjecture \ref{QRconj2}.
 Then,  
\begin{eqnarray} \quad \label{average1}
\frac{1}{|L_n^{(r_2)}(X)|} \sum_{K \in L_n^{(r_2)}(X)} n_{K,C} = \sum_{ q } \frac{q(1-|C|/|S_n|+f(q)) }{1+f(q)}\prod_{p<q} \frac{|C|/|S_n|}{1+f(p)}+O \left( \frac{1}{\log X}\right). 
\end{eqnarray}
\end{Thm}

For $S_3$-fields, Martin and Pollack \cite{MP} computed $(\ref{average1})$ for $C=e, [(123)]$. The main key ingredient was counting 
$S_3$-fields with finitely many local conditions, which is a recent result of Taniguchi and Thorne \cite{TT}. In \cite{CK3}, we were able to count $S_4$ and $S_5$-fields with finitely many local conditions using a result of Belabas, Bhargava and Pomerance \cite{BBP}, and a result of Shankar and Tsimerman \cite{ST}. 

Our key idea is to use the unique quadratic subextension $F=\Bbb Q[\sqrt{d_K}]$, which we call the quadratic resolvent. For unconditional bounds on $n_{K,C}$, we use the inequality $n_{K,C} \leq n_{F,1} \leq N_{F,-1}$ or $n_{K,C} \leq n_{F,-1} \leq N_{F,1}$  depending on whether $C\subset A_n$ or $C\not\subset A_n$. 

We have unconditional bounds of $N_{F,\pm 1}$ by Norton \cite{No} and Pollack \cite{P2014}. We review this in Section \ref{UB}.

By using the zero-free region of $L(s,\chi_F)$, we can get conditional bounds on $n_{K,C}$. This is done in Section \ref{CB}.
We need to count the number of $S_n$-fields with the same quadratic resolvent. For $S_3$-fields, we can estimate such numbers by using the result of \cite{CT}. This is done in Section \ref{QR}. 
But for $n\geq 4$, we do not have such a result. So we state it as 
Conjecture \ref{QRconj2}. In the Appendix, using the result in \cite{CT1}, we count the number of $S_4$-fields with the given cubic resolvent.

In section \ref{AV}, we establish $(\ref{average1})$ under the counting conjectures 
$(\ref{estimate})-(\ref{estimate1})$ and Conjecture \ref{QRconj2}. See Theorem \ref{main1} and tables below it for the average values which were computed using PARI. 

In Section \ref{Counting}, we explain counting number fields with finitely many local conditions, which is the main tool for the proof. 

In Section \ref{AA}, we give another proof for $S_n$-fields, $n=3,4,5$, in order to avoid using Conjecture \ref{QRconj2}. 
We use the strong Artin conjecture and zero-free regions of different Artin $L$-functions for each conjugacy class.

In Section \ref{N}, we consider the average of $N_{K,C}$, the smallest prime $p$ with Frob$_p \in C$. In Martin and Pollack \cite{MP}, the average of $N_{K,C}$ for $S_3$-fields was studied under GRH for the Dedekind zeta functions. We generalize their result and compute the averages of $N_{K,C}$ under GRH for Dedekind zeta functions and the counting conjectures $(\ref{estimate})-(\ref{estimate1})$. Tables for the average values for $S_3$, $S_4$, and $S_5$ are provided.

In Section \ref{N-b}, we obtain an unconditional result for the average of $N_{K,C'}$ under Conjecture \ref{QRconj2} where $C'$ is the union of all the conjugacy classes not in $A_n$. Hence in particular, for $S_3$-fields and $C=[(12)]$, we have an unconditional result for the average of $N_{K,C}$.

\subsection*{Acknowledgments} We thank G. Henniart for his help on the conductor of Artin $L$-functions. We thank P. Pollack for helpful discussions.

\section{Counting number fields with local conditions} \label{Counting}

Let $K$ be a $S_n$-field for $n\geq 3$. Let $\mathcal S = (\mathcal{LC}_p)$ be a finite set of local conditions: $\mathcal{LC}_p=S_{p,C}$ means that $p$ is unramified and the conjugacy class of Frob$_p$ is $C$.  
Define $|\mathcal S_{p,C}|=\frac {|C|}{|S_n|(1+f(p))}$ for some function $f(p)$ which satisfies
$f(p)=O(\frac 1p)$. There are also several splitting types of ramified primes, which are denoted by $r_1,r_2,\dots,r_w$: $\mathcal{LC}_p=S_{p,r_j}$ means that $p$ is ramified and its splitting type is $r_j$. We assume that there are positive valued functions $c_1(p)$, $c_2(p)$, $\dots$, $c_w(p)$ with $\sum_{i=1}^w c_i(p)=f(p)$ and define $|S_{p,r_i}|=\frac{c_i(p)}{1+f(p)}$. 
Let $|\mathcal S|=\prod_p |\mathcal{LC}_p|$. 
\begin{Conj} \label{Conj-estimate}
Let $L_n^{(r_2)}(X;\mathcal S)$ be the set of $S_n$-fields $K$ of signature $(r_1,r_2)$ with $|d_K| < X$ and the local conditions $\mathcal S$. 
\begin{eqnarray} 
|L_n^{(r_2)}(X)| &=& A(r_2) X +O(X^{\delta}),\label{estimate}\\
|L_n^{(r_2)}(X;\mathcal S)| &=& |\mathcal S| A(r_2) X + O\left(\Big(\prod_{p \in S} p\Big)^{\gamma} X^{\delta} \right) \label{estimate1}
\end{eqnarray}
for some positive constant $\delta<1$ and $\gamma$, and the implied constant is uniformly bounded for $p$ and local conditions at $p$. 
\end{Conj}
This conjecture is true for $S_3, S_4$ and $S_5$-fields.  For $S_3$-fields, we use a result of Taniguchi and Thorne \cite{TT}.  
Let $f(p)=p^{-1}+p^{-2}$. Put
$$|S_{p,r_j}|=\frac{p^{-1}}{1+f(p)},\:\frac{p^{-2}}{1+f(p)},
$$
for $r_j=(1^21),  (1^3),$ respectively. 
Then
\begin{Thm} \cite{TT} 
Let $D_0=\frac {C_1 C^0}{12\zeta(3)}$, $D_1=\frac{C_1C^1}{12\zeta(3)}$.
\begin{eqnarray*}
|L_3^{(r_2)}(X,\mathcal S)| = |\mathcal S| D_{r_2} X+ O_{\epsilon}\left(\bigg(\prod_{p\in \mathcal S} p \bigg)^{\frac {16}9} X^{\frac 79+\epsilon }\right).
\end{eqnarray*}
\end{Thm}
 
For $S_4$-fields, take $f(p)=p^{-1}+2p^{-2}+p^{-3}$. For a conjugacy class $C$ of $S_4$, 
let $$|S_{p,C}|=\frac{|C|}{24(1+f(p))}.$$
Put
$$|S_{p,r_j}|=\frac{1/2 \cdot 1/p}{1+f(p)},\:\frac{1/2 \cdot 1/p}{1+f(p)},\:\frac{1/2 \cdot 1/p^2}{1+f(p)},\:\frac{1/2 \cdot 1/p^2}{1+f(p)},\:\frac{1/p^2}{1+f(p)},\: \mbox{and }\frac{1/p^3}{1+f(p)}
$$
for $r_j=(1^211),(1^22),(1^21^2),(2^2),(1^31), (1^4),$ respectively. By using the results of \cite{BBP}, \cite{yang}, we showed

\begin{Thm} \cite{CK3} 
Let $D_i=d_i \prod_{p} ( 1+p^{-2} -p^{-3} - p^{-4})$, and
$d_0=\frac{1}{48},d_1=\frac{1}{8},$ and $d_0=\frac{1}{16}$.
\begin{eqnarray*}
|L_4^{(r_2)}(X,\mathcal S) |= |\mathcal S| D_{r_2} X+ O_{\epsilon}\left(\bigg(\prod_{p\in \mathcal S} p \bigg)^2 X^{\frac{143}{144}+\epsilon }\right).
\end{eqnarray*}
\end{Thm}

For $S_5$-fields, take $f(p)=p^{-1}+2p^{-2}+2p^{-3}+p^{-4}$.  For a conjugacy class $C$ of $S_5$, 
let $$|S_{p,C}|=\frac{|C|}{120(1+f(p))}.$$

Put 
\begin{eqnarray*}
|S_{p,r_j}| & =& \frac{1/6 \cdot 1/p}{1+f(p)},\:\frac{1/2 \cdot 1/p}{1+f(p)},\:\frac{1/3 \cdot 1/p}{1+f(p)},\:\frac{1/2 \cdot 1/p^2}{1+f(p)},\: \frac{1/2 \cdot 1/p^2}{1+f(p)},\:\frac{1/2 \cdot 1/p^2}{1+f(p)},\frac{1/2 \cdot 1/p^2}{1+f(p)},\\
 & & \frac{1/p^3}{1+f(p)},\:\frac{1/p^3}{1+f(p)},\:\mbox{ and } \frac{1/p^4}{1+f(p)}
\end{eqnarray*}
for $r_j=(1^2111),(1^212),(1^23),(1^21^21),(2^2 1),(1^3 11),(1^3 2),(1^3 1^2), (1^4 1), (1^5),$ respectively.

By using the result of \cite{ST}, we showed  

\begin{Thm} \cite{CK3} 
Let $D_i=d_i \prod_p (1+ p^{-2} -p^{-4} -p^{-5})$ and $d_0,d_1,d_2$ are $\frac{1}{240}, \frac{1}{24}$ and $\frac{1}{16}$, respectively.
\begin{eqnarray*}
|L_5^{(r_2)}(X,\mathcal S)|= |\mathcal S| D_{r_2} X+ O_{\epsilon}\left(\bigg(\prod_{p\in \mathcal S} p \bigg)^{2-\epsilon} X^{\frac{199}{200}+\epsilon }\right).
\end{eqnarray*}
\end{Thm}

\section{Bounds on $n_{K,C}$}

Recall that $n_{K,C}$ is the smallest prime which is ramified in $K$ or for which Frob$_p$ does not belong to $C$. Now $\widehat{K}$ has the quadratic field 
$F$ fixed by $A_n$, i.e., $F=\Bbb Q[\sqrt{d_K}]$. Let $d_F$ be the discriminant of $F$. Then clearly, $|d_F|\leq |d_K|$.
By abuse of language, we call such $F$ the quadratic resolvent of $K$.

If $C \subset A_n$ and $Frob_p \in C$, then $p$ splits in $F$. Hence $n_{K,C}\leq n_{F,1}$.

If $C \not\subset A_n$ and $Frob_p \in C$, then $p$ is inert in $F$. Hence $n_{K,C}\leq n_{F,-1}$, 

\subsection{Unconditional bounds of $n_{K,C}$} \label{UB}

By Norton \cite{No}, $N_{F,-1}\ll_\epsilon |d_F|^{\frac{1}{4\sqrt{e}}+\epsilon} \ll_\epsilon |d_K|^{\frac{1}{4\sqrt{e}}+\epsilon}$. Since $n_{F,1} \leq N_{F,-1}$, 
\begin{eqnarray*}
n_{K,C}  \ll_\epsilon |d_F|^{\frac{1}{4\sqrt{e}}+\epsilon} \ll_\epsilon |d_K|^{\frac{1}{4\sqrt{e}}+\epsilon} \mbox{ for $C \subset A_n$.}
\end{eqnarray*}

By Pollack \cite{P2014},  $N_{F,1} \ll_\epsilon |d_F|^{\frac 14+\epsilon} \ll_\epsilon |d_K|^{\frac 14+\epsilon}$. Since $n_{F,-1} \leq N_{F,1}$,
\begin{eqnarray*}
n_{K,C} \ll_\epsilon |d_F|^{\frac 14+\epsilon} \ll_\epsilon |d_K|^{\frac 14+\epsilon} \mbox{ for $C \not\subset A_n$.}
\end{eqnarray*}

\subsection{Conditional bounds of $n_{K,C}$} \label{CB}
We obtain conditional bounds on $n_{F,C}$ under the zero-free region of $L(s,\chi_F)$, where $\chi_F(p)=(\frac {d_F}p)$.
Suppose $L(s,\chi_F)$ is zero free on $[\alpha,1]\times [-(\log |d_F|)^2, (\log |d_F|)^2]$. Then by \cite{CK2},
$$
-\frac{L'}{L}(\sigma, \chi_F) = \sum_{ p < ( \log |d_F|)^{16/(1-\alpha)}}  \frac{\chi_F(p)\log p}{p^\sigma} + O(1),
$$
for $1 \leq \sigma \leq 3/2$. Hence, it implies that
$$
\left| -\frac{L'}{L}(\sigma, \chi_F) \right|  \leq \frac{16}{(1-\alpha)}\log \log |d_F| + O(1).
$$

Now for $C \subset A_n$, consider $\zeta_F(s)=\zeta(s)L(s,\chi_F)$. We obtain
$$\sum_p \frac {(1+\chi_F(p))\log p}{p^{\sigma}}=\frac 1{\sigma-1}-\frac {L'}L(\sigma,\chi_F)+O(1).
$$

 For a while, we assume that $n_{K,C} \geq 3$. For each prime $p < n_{K,C}$, the prime $p$ splits in the quadratic resolvent $F$. (i.e., $\chi_F(p)=1$ for all $p < n_{K,C}$.) 

Then
$$\sum_{p<n_{K,C}} \frac {2 \log p}{p^{\sigma}}\leq \frac 1{\sigma-1}-\frac {L'}L(\sigma,\chi_F)+O(1).
$$
By taking $\sigma-1=\frac {\lambda}{\log n_{K,C}}$, we have
$$
\frac {1-2e^{-\lambda}}{2\lambda}\log n_{K,C}\leq \frac{16}{(1-\alpha)}\log \log |d_F| + O(1).
$$
Hence 
$$n_{K,C}\ll (\log |d_F|)^{\frac {16}{(1-\alpha)A}}\ll (\log |d_K|)^{\frac {16}{(1-\alpha)A}},
$$
where $A=\mbox{sup}_{\lambda\geq 0} \frac {1-2e^{-\lambda}}{\lambda}$, which is 0.37... when $\lambda=1.678...$ 
When $n_{K,C}=2$, it clearly satisfies the above inequality. Hence we can remove the assumption $n_{K,C}\geq 3$. 

Now for $ C \not\subset A_n$, consider 
$$\sum_p \frac {(1-\chi_F(p))\log p}{p^{\sigma}}=\frac 1{\sigma-1}+\frac {L'}L(\sigma,\chi_F)+O(1).
$$
For each prime $p<n_{K,C}$, the prime $p$ is inert in $F$. (i.e., $\chi_F(p)=-1)$ for all $p<n_{K,C}$.
We have the inequality
$$\sum_{p<n_{K,C}} \frac {2\log p}{p^{\sigma}}\leq \frac 1{\sigma-1}+\frac {L'}L(\sigma,\chi_F)+O(1).
$$
Hence 
$$n_{K,C}\ll (\log |d_F|)^{\frac {16}{(1-\alpha)A}} \ll (\log |d_K|)^{\frac {16}{(1-\alpha)A}}. 
$$
(Here we implicitly assume that $n_{K,C} \geq 3$ and remove the restriction after we obtain the upper bound.)

Let 
$$L(X)^{\pm}=\{ F|\, F:\mbox{quadratic field, $\pm d_F \leq X$} \}.
$$
We may treat $L(X)^{\pm}$ as families of quadratic Dirichlet $L$-functions $L(s,\chi_F)$.
By applying Kowalski-Michel's theorem \cite{KM} to $L(X)^\pm$, 
we can show that every $L$-function in $L(X)^\pm$ is zero-free on $[\alpha,1]\times[-(\log X)^2, (\log X)^2]$ except for $O(X^{\beta_n/2})$ $L$-functions. Here $\beta_n$ is some small constant which appears in Conjecture \ref{QRconj2}.  Here $\alpha$ is a fixed constant close to $1$ depending on $\beta_n$ but independent of $X$. For the detail of how to apply Kowalski-Michel's theorem, we refer to \cite{CK2}. 

\section{Number fields with the same quadratic resolvent} \label{QR}

Let $F$ be a quadratic field.  Then there are infinitely many $S_n$-fields having the same quadratic resolvent $F$. Let $QR_n(X,F)$ be the set of $S_n$-fields with the common quadratic resolvent $F$ and the absolute value of the discriminant bounded by $X$.

Since $d_K=d_Fm^2$ for some $m \in \mathbb{Z}^+$, by the counting conjecture (\ref{estimate}), it is expected that $|QR_n(X,F)| \ll (X/|d_F|)^{1/2+\epsilon}$.  

For our purpose, a weaker form is enough.

\begin{Conj} \label{QRconj2} There is a constant $\beta_n$ with $ 0< \beta_n < 1/2$ for which
\begin{eqnarray*}
|QR_n(X,F)| \ll \left( \frac{X}{|d_F|} \right)^{1-\beta_n},
\end{eqnarray*}
where the implied constant is independent of $F$.
\end{Conj}

We prove it for $n=3$. When $n=4$, we are not able to prove it. However, in the Appendix, we obtain a good bound on the number of $S_4$-fields with the given cubic resolvent.

Given a quadratic field $F$, let
$$\Phi_F(s)=\frac 12+\sum_{K\in \mathcal F(F)} \frac 1{f(K)^s},
$$
where $d_K=d_F f(K)^2$, and $\mathcal F(F)$ is the set of all cubic fields $K$ with the quadratic resolvent field $F$. 

Cohen and Thorne \cite{CT} found an explicit expression for $\Phi_F(s)$. Let $D$ be a fundamental discriminant, and $D^*=-3D$ if $3\nmid D$, and $D^*=-\frac D3$ if $3|D$. Define $\mathcal L_3(D)=\mathcal L_{D^*}\cup \mathcal L_{-27D}$, where $\mathcal L_N$ is the set of cubic fields of discriminant 
$N$. 

By Theorem 2.5 in \cite{CT},
$$
\Phi_F(s)=\sum_{i=1}^{|\mathcal L_3(d_F)+1|}\Phi_i(s), \quad \Phi_i(s)=\sum_{n=1}^\infty \frac{a_i(n)}{n^s}, 
$$
and $a_i(n) \leq 2^{\omega(n)} \ll 2^{\frac{\log n}{\log \log n}} \ll n^{\epsilon}$. Hence each $\Phi_i(s)$ is absolutely convergent for $Re(s)>1$. Also Theorem 2.5 in \cite{CT} implies 
$$
\Phi_i(1+c+ it) \ll \left( \frac{\zeta(1+c)}{\zeta(2+2c)}\right)^2 \ll \frac{1}{c^2}. 
$$

Now we apply Perron's formula to each $\Phi_i(s)$. For $c=\frac{1}{\log x}$,

$$\sum_{ n < x} a_i(n)=\int_{1+c-i x}^{1+c+ix} \Phi_i(1+c+it) \frac {x^{s}}{s}\, ds+O\left(x^\epsilon\right),
$$
with an absolute implied constant.  Since $\Phi_i(1+c+ it) \ll (\log x)^2$, the integral is majorized by $x(\log x)^3$. So
$$
\sum_{ n < x} a_i(n) \ll x(\log x)^3.
$$

By \cite{EV}, $|\mathcal L_N|\ll |N|^{\frac 13+\epsilon}$. Hence
$$
|\{K\in \mathcal F(F)| f(K) \leq x \}| = \sum_{i=1}^{|\mathcal L_3(d_F)+1|} \sum_{n < x} a_i(n) \ll |d_F|^{\frac 13 + \epsilon}x(\log x)^3.
$$
 Therefore,
$$|\{K\in \mathcal F(F)\, |\, f(K)\leq x\}|
\ll x(\log x)^3 |d_F|^{\frac 13+\epsilon}.
$$
Hence we have proved
\begin{Prop}
$$|QR_3(X,F)|\ll X^{\frac 12} (\log X)^3 |d_F|^{-\frac 16+\epsilon},
$$
with an absolute implied constant.
\end{Prop}

\section{Average value of $n_{K,C}$} \label{AV}
In this section, we prove Theorem \ref{main} under Conjecture \ref{QRconj2}, and also under the counting conjectures $(\ref{estimate})-(\ref{estimate1})$.

For simplicity of notation, we denote $L_n^{(r_2)}(X)$ by $L(X)$. Take $y=\frac{1-\delta}{4\gamma} \log X$, where $\delta$ and $\gamma$ are the constants in $(\ref{estimate})$ and $(\ref{estimate1})$. Then
\begin{eqnarray*}
\sum_{K \in L(X)} n_{K,C} = \sum_{K \in L(X),\: n_{K,C} \leq y } n_{K,C} + \sum_{K \in L(X),\: n_{K,C} > y} n_{K,C}.
\end{eqnarray*}

Here $n_{K,C}=q$ means that for all primes $p < q$, Frob$_p \in C$ and $q$ is ramified or Frob$_q \not\in C$. By the counting conjectures, there are
$$
\frac{1-|C|/|S_n|+f(q) }{1+f(q)}\prod_{p<q} \frac{|C|/|S_n|}{1+f(p)} A(r_2)X + O(X^{\frac{1+3\delta}{4}})
$$
such number fields in $L(X)$. Hence,
\begin{eqnarray*}
\sum_{K \in L(X),\: n_{K,C} \leq y} n_{K,C}& =& \sum_{q \leq y} q \sum_{K \in L(X),\: n_{K,C}=q} 1 \\
   &=& A(r_2)X \sum_{ q \leq y} \frac{q(1-|C|/|S_n|+f(q))}{1+f(q)}\prod_{p<q} \frac{|C|/|S_n|}{1+f(p)} + O(y^2X^{\frac{1+3\delta}{4}})\\
   &=& A(r_2)X \sum_{ q } \frac{q(1-|C|/|S_n|+f(q)) }{1+f(q)}\prod_{p<q} \frac{|C|/|S_n|}{1+f(p)}\\
 &+ & O\left( X\sum_{q > y} q \prod_{p < q} |C|/|S_n| + (\log X)^2 X^{\frac{1+3\delta}{4}}\right).
\end{eqnarray*}
Since
\begin{eqnarray*}
\sum_{q > y} q \prod_{p < q} |C|/|S_n|  \leq \sum_{q > y} q \left(\frac {|C|}{|S_n|}\right)^{\pi(q)} \ll \frac{1}{\log X},
\end{eqnarray*}
\begin{eqnarray*}
\sum_{K \in L(X),\: n_{K,C} \leq y } n_{K,C}=A(r_2)X \sum_{ q } \frac{q(1-|C|/|S_n|+f(q)) }{1+f(q)}\prod_{p<q} \frac{|C|/|S_n|}{1+f(p)}+O\left( \frac{X}{\log X} \right). 
\end{eqnarray*}

Now we divide the sum $\sum_{|d_K| \leq X, \: n_{K,C} > y } n_{K,C}$ into two subsums. Let $E(X)$ be the set of $S_n$-fields in $L(X)$ for which 
the quadratic $L$-function $L(s,\chi_F)$, where $F$ is the quadratic resolvent, may not have the desired zero-free region in Section \ref{CB}.
\begin{equation} \label{error}
\sum_{K \in L(X),\: n_{K,C} > y } n_{K,C} = \sum_{n_{K,C} > y,\:  K \not\in E(X)} n_{K,C} + \sum_{n_{K,C} > y,\: K \in E(X)} n_{K,C}.
\end{equation}

Let's deal with the second sum. From the unconditional bound in Section \ref{UB}, $n_{K,C} \ll  |d_F|^{\frac{1}{4\sqrt{e}}+\epsilon}$ or $ |d_F|^{\frac 14+\epsilon}$  depending on whether $C \subset A_n$ or $C \not\subset A_n$.  In any case $n_{K,C} \ll |d_F|^{\frac 14+\epsilon}$. For such $F$, by Conjecture \ref{QRconj2}, we have at most $(X/|d_F|)^{1-\beta_n}$ $S_n$-fields with the same quadratic resolvent $F$. Hence such $F$ contributes at most $|d_F|^{\frac 14+\epsilon}(X/|d_F|)^{1-\beta_n} \ll X^{1-\beta_n}$. In Section \ref{CB}, we showed that there are at most 
$X^{\beta_n/2}$ quadratic fields $F$ which may not have the desired zero-free region. Therefore,
\begin{eqnarray*}
 \sum_{n_{K,C} > y,\: K \in E(X) } n_{K,C} \ll X^{\beta_n/2} X^{1-\beta_n} \ll X^{1-\beta_n/2}.
\end{eqnarray*}   

To handle the first sum, we use the conditional bound on $n_{K,C}$ in Section \ref{CB}.
\begin{eqnarray*}
\sum_{n_{K,C} > y,\: K \not\in E(X)} n_{K,C} &\ll&  (\log X)^{\frac{16}{(1-\alpha)A}} \sum_{ n_{K,C} > y} 1 \\
                                                                                     &\ll& (\log X)^{\frac{16}{(1-\alpha)A}} \left( X \prod_{p<y} |C|/|S_n| + X^{\delta}\left( \prod_{p<y} p \right)^\gamma \right) \\
&\ll& X(\log X)^{\frac{16}{(1-\alpha)A}}\left(\frac {|C|}{|S_n|}\right)^{\pi(y)}+ X^{\frac{1+\delta}{2}} .
\end{eqnarray*}
We use the fact that $\left(\frac {|C|}{|S_n|}\right)^{\pi(y)}\ll e^{-\frac {\log X}{\log\log X}}\ll (\log X)^{-k}$ for any $k$.
Our discussion is summarized as follows:
\begin{Thm} \label{main1}
Let $L_n^{(r_2)}(X)$ be the set of $S_n$-fields $K$ of signature $(r_1,r_2)$ with $|d_K|<X$. Assume the counting conjectures $(\ref{estimate})-(\ref{estimate1})$ and Conjecture \ref{QRconj2}. Let $C$ be a conjugacy class of $S_n$ and $n_{K,C}$ be the least prime with Frob$_p \not\in C$. 
 Then,  
\begin{eqnarray}\quad \label{MT}
\frac{1}{|L_n^{(r_2)}(X)|} \sum_{ K \in L_n^{(r_2)}(X)} n_{K,C} = \sum_{ q } \frac{q(1-|C|/|S_n|+f(q)) }{1+f(q)}\prod_{p<q} \frac{|C|/|S_n|}{1+f(p)}+O \left( \frac{1}{\log X}\right).
\end{eqnarray}
\end{Thm}

For $S_3$-fields, the counting conjectures $(\ref{estimate})-(\ref{estimate1})$ and Conjecture \ref{QRconj2} are true. 
Hence, the above theorem holds unconditionally. For $S_4$ and $S_5$-fields, the counting conjectures $(\ref{estimate})-(\ref{estimate1})$ are true. Hence
under Conjecture \ref{QRconj2}, Theorem \ref{main1} holds for $S_4$ and $S_5$-fields.

The tables below show average values of $n_{K,C}$ for $S_3$, $S_4$ and $S_5$-fields. The computations are done by PARI.
\begin{center}
\begin{tabular}{|c|c|}
  \hline
  $S_3$ & \mbox{Average of $n_{K,C}$} \\
\hline
$[e]$ & $2.1211027...$ \\
\hline
$[(12)]$ & $2.6719625...$\\
\hline
$[(123)]$ & $2.3192802...$\\
\hline 
\end{tabular}
\quad
\begin{tabular}{|c|c|}
  \hline
  $S_4$ & \mbox{Average of $n_{K,C}$} \\
\hline
$[e]$ & $2.0206694...$ \\
\hline
$[(12)(34)]$ & $2.0691556...$\\
\hline
$[(1234)]$ & $2.1653006...$\\
\hline 
$[(12)]$ & $2.1653006...$\\
\hline 
$[(123)]$ & $2.2516575...$\\
\hline 
\end{tabular}
\quad
\begin{tabular}{|c|c|}
  \hline
  $S_5$ & \mbox{Average of $n_{K,C}$} \\
\hline
$[e]$ & $2.0036404...$ \\
\hline
$[(12)(34)]$ & $2.0632551...$\\
\hline
$[(123)]$ & $2.0891619...$\\
\hline 
$[(12)(345)]$ & $2.0891619...$\\
\hline 
$[(12)]$ & $2.0399630...$\\
\hline 
$[(1234)]$ & $2.1505010...$\\
\hline 
$[(12345)]$ & $2.1120340...$\\
\hline
\end{tabular}
\end{center}

\begin{Rmk}
From the tables above, we can see that the average value of $n_{K,C}$ is close to $2$ and $n_{K,C} < n_{K,C'}$ if $|C|<|C'|$. In fact, it is expected from the formula for the average value of $n_{K,C}$. The probability for $n_{K,C}$ to be $2$ is $\frac{1-|C|/|S_n|+f(2)}{1+f(2)}$, which happens to most of the number fields. For example, for $S_5$-fields, the probability for $n_{K,[e]}$ to be $2$ is $0.996396...$
\end{Rmk}

\begin{Rmk} 
Let $L(X)^\pm$ be the set of real/complex quadratic extension $F$ with $\pm d_F \leq X$. For the sake of completeness, we record the average of $n_{F,\pm 1}.$ 
It is easy to check that the probabilities for a prime $p$ is to ramify, split, or be inert are $\frac{1}{p+1}, \frac{p}{2(p+1)}$, or $\frac{p}{2(p+1)}$ respectively.  Hence, 

\begin{eqnarray*}
\lim_{ X \rightarrow \infty} \frac{\sum_{\pm d_F \leq X} n_{F,\pm 1} }{|L(X)^\pm|} &=& \sum_{q} \frac{q^2+2q}{2(q+1)}\prod_{p < q} \frac{p}{2(p+1)}= 2.83264\dots. 
\end{eqnarray*}
\end{Rmk}

\section{Alternative Proof of (\ref{MT}) without Conjecture \ref{QRconj2}} \label{AA}

In this section, we show how we can avoid using Conjecture \ref{QRconj2} which is necessary to estimate (\ref{error}). Instead we assume the strong Artin conjecture and use various Artin $L$-functions, depending on the conjugacy class. 
We consider $S_n$-fields for $n=3,4,5$. Since the strong Artin conjecture is known for $S_3$ and $S_4$-fields \cite{C}, our result is unconditional for $S_3$ and 
$S_4$-fields. We show, by a case by case analysis on each $C$,
\begin{eqnarray*}
\sum_{ K \in L(X),\: n_{K,C} > y} n_{K,C} =O\left(\frac X{\log X}\right).
\end{eqnarray*}

We still divide it into two subsums
$$\sum_{K \in L(X),\: n_{K,C} > y } n_{K,C} = \sum_{n_{K,C} > y,\: K \not\in E(X)} n_{K,C} + 
\sum_{n_{K,C} > y,\: K \in E(X)} n_{K,C}.
$$
However, the exceptional set $E(X)$ will be different for each $C$, since we consider zero-free regions of different Artin $L$-functions.
The second sum is estimated by using the unconditional bounds of $n_{K,C}$ in Section \ref{UB}.
For the first sum, we need conditional bounds, conditional on zero-free regions of various Artin $L$-functions.
 We use the following formula as in \cite{LMO}:
For a conjugacy class $C$ of $S_n$, define, for $\sigma >1$,
\begin{eqnarray} \label{F_C}
F_C(\sigma)=-\frac{|C|}{|S_n|}\sum_{\psi}\overline{\psi}(C)\frac{L'}{L}(\sigma, \psi, \widehat{K}/\Q),
\end{eqnarray}
where $\psi$ runs over the irreducible characters of $S_n$ and $L(s,\psi, \widehat{K}/\Q)$ is the Artin $L$-function attached to the character $\psi$. By orthogonality of characters, 
\begin{eqnarray} \label{F_C2}
F_C(\sigma)=\sum_{p}\sum_{m=1}^\infty \frac{\theta(p^m) \log p}{ p^{m\sigma}},
\end{eqnarray}
where for a prime $p$ unramified in $\widehat{K}$,
$$
\theta(p^m)=\left\{ \begin{array}{cc} 1 & \mbox{ if } (\mbox{Frob}_p)^m \in C, \\
                                                             0 & \mbox{ otherwise.}
\end{array}  \right.
$$
and $0 \leq \theta(p^m) \leq 1$ if $p$ ramifies in $\widehat{K}$. 

\subsection{$S_4$-fields}
Here, we follow the notations in \cite{DF} for characters of $S_4$. 

\subsubsection{Case 1. $C=(1234)$ } From (\ref{F_C2}),
$$\sum_{p\in C} \frac {\log p}{p^\sigma}=\frac 14\cdot \frac 1{\sigma-1}- \frac 14 \left(-\frac {L'}{L}(s,\chi_2)-
\frac {L'}{L}(\sigma,\chi_4)+\frac {L'}{L}(\sigma,\chi_5)\right)+O(1).
$$
Here 
$$-\frac {L'}{L}(\sigma,\chi_2)-\frac {L'}{L}(\sigma,\chi_4)=\sum_{p} \frac {\chi_2(p)+\chi_4(p)}{p^\sigma}+O(1)\geq -2\sum_{p\in C} \frac {\log p}{p^\sigma}+O(1).
$$

Hence we have
\begin{eqnarray} \label{F_C3}
\frac 12 \sum_{p\in C} \frac {\log p}{p^\sigma}\leq \frac 14\cdot \frac 1{\sigma-1}- \frac 14 \frac {L'}{L}(\sigma,\chi_5)+O(1).
\end{eqnarray}

Since $\chi_5=\chi_4\otimes\chi_2$, the conductor of $\chi_5$ is at most $|d_K|^3$. Since $\chi_4$ is modular \cite{C}, $\chi_5$ is modular, i.e., $L(s,\chi_5)$ is a cuspidal automorphic $L$-function of $GL_3/\Q$. Consider a family of Artin $L$-functions:
\begin{eqnarray*}
L(X) = \{L(s, \chi_5)\, | \, K \in L(X)^{r_2} \}.
\end{eqnarray*}

Then all $L$-functions in $L(X)$ are distinct because $L(s,\chi_4)$'s are distinct since $K$ is arithmetically solitary. (See \cite{CK2} for the detail.) By applying Kowalski-Michel's theorem to $L(X)$, every $L$-function in the family is zero-free on $[\alpha,1] \times [-(3\log X)^2, (3\log X)^2]$ except for $O(X^{1/1000})$ $L$-functions.  

For a $L$-function with such a zero-free region, 
\begin{eqnarray} \label{Ineq1}
-\frac{L'}{L}(\sigma, \chi_5) \leq \frac{16\cdot 3}{(1-\alpha)} \log\log d_K + O(1), 
\end{eqnarray}
for $1\leq \sigma \leq 3/2.$ (See (5.1) in \cite{CK2}).  Plugging (\ref{Ineq1}) into (\ref{F_C3}), and taking $\sigma=1+\frac{\lambda}{\log n_{K,C}}$, we obtain
$\frac{1-2e^{-\lambda}}{4\lambda} \log n_{K,C} \leq  \frac{12}{(1-\alpha)} \log\log |d_K| + O(1).$
Hence, 
\begin{eqnarray}
n_{K,C} \ll (\log |d_K|)^{\frac{48}{(1-\alpha)A}},
\end{eqnarray}
where $A=\sup_{\lambda \geq 0}\frac{1-2e^{-\lambda}}{\lambda}$. 

Let $E(X)$ be the exceptional subset in $L(X)$. Then $|E(X)| \ll X^{1/1000}$. By abuse of language, 
$K \not\in E(X)$ means $L(s,\chi_5) \not\in E(X)$.  Then
\begin{eqnarray*}
\sum_{ K \in L(X), \: n_{K,C} > y} n_{K,C}  &=& \sum_{ K \not \in E(X),\: n_{K,C} > y } n_{K,C}+ 
\sum_{ K \in E(X),\: n_{K,C} > y} n_{K,C} \\
  &\ll& (\log X)^{\frac{48}{(1-\alpha)A}} \sum_{ K \in L(X)^{r_2},\: n_{K,C} > y} 1 + X^{1/4+\epsilon}\cdot X^{1/1000}\\
&\ll&  (\log X)^{\frac{48}{(1-\alpha)A}} \left(X \left(\frac{|C|}{|S_4|}\right)^{\pi(y)}+X^{\frac{1+\delta}{2}}\right)\ll \frac {X}{\log X}.
\end{eqnarray*}

\subsubsection{Case 2. $C=(12)(34)$}  From (\ref{F_C2})
$$\sum_{p\in C} \frac {\log p}{p^\sigma}=\frac 18\cdot \frac 1{\sigma-1}- \frac 18 \left(\frac {L'}{L}(\sigma,\chi_2)+2\frac {L'}L(\sigma,\chi_3)-
\frac {L'}{L}(\sigma,\chi_4)-\frac {L'}{L}(\sigma,\chi_5)\right)+O(1).
$$
Since  
$$-\frac {L'}{L}(\sigma,\chi_2)-2\frac {L'}{L}(\sigma,\chi_3)\leq 5\sum_{p} \frac {\log p}{p^\sigma}+O(1)=\frac 5{\sigma-1}+O(1),
$$
we have
$$ \sum_{p\in C} \frac {\log p}{p^\sigma}\leq \frac 68\cdot \frac 1{\sigma-1}+\frac 18\left(\frac {L'}{L}(\sigma,\chi_4)+\frac {L'}{L}(\sigma,\chi_5)\right)+O(1).
$$

Now consider a family of Artin $L$-functions:
$$
L(X)=\{ L(s, \chi_4)L(s,\chi_5)\, |\, K \in L(X)^{r_2} \}.
$$
We apply Kowalski-Michel's theorem to $L(X)$. Then every $L$-function in $L(X)$ is zero-free on 
$[\alpha ,1] \times [-( 4\log |d_K|)^2, (4\log |d_K|)^2]$ except for $O(X^{1/1000})$ $L$-functions. Since $L(s,\chi_4)$ and $L(s,\chi_5)$ are simultaneously zero-free on $[\alpha ,1] \times [-( 4\log |d_K|)^2, (4\log |d_K|)^2]$,
\begin{eqnarray*}
\left| -\frac{L'}{L}(\sigma, \chi_4 ) \right|,\quad \left| -\frac{L'}{L}(\sigma, \chi_5 ) \right| \leq \frac{16\cdot 3}{(1-\alpha)} \log\log |d_K| + O(1),
\end{eqnarray*} 
and with this conditional bound, as we did in the previous section, we can show
\begin{eqnarray*}
\sum_{ K \in L(X),\: n_{K,C} > y} n_{K,C} =O\left(\frac X{\log X}\right).
\end{eqnarray*}

\subsubsection{Case 3. $C=(12)$}  From (\ref{F_C2}),

$$\sum_{p\in C} \frac {\log p}{p^\sigma}=\frac 14\cdot \frac 1{\sigma-1}- \frac 14 \left(-\frac {L'}{L}(\sigma,\chi_2)+
\frac {L'}{L}(\sigma,\chi_4)-\frac {L'}{L}(\sigma,\chi_5)\right)+O(1).
$$
Here 
$$-\frac {L'}{L}(\sigma,\chi_2)-\frac {L'}{L}(\sigma,\chi_5)\geq -2\sum_{p\in C} \frac {\log p}{p^\sigma}+O(1).
$$
Hence we have
$$\frac 12 \sum_{p\in C} \frac {\log p}{p^\sigma}\leq \frac 14\cdot \frac 1{\sigma-1}-\frac 14\cdot \frac {L'}{L}(\sigma,\chi_4)+O(1).
$$

Let $\tilde L(X)=\{ L(s,\chi_4)\, |\, K \in L(X)\}$ and apply Kowalski-Michel's theorem to $\tilde L(X)$. Then every $L(s,\chi_4)$ in $\tilde L(X)$ except for 
$O(X^{1/1000})$ $L$-functions satisfies
\begin{eqnarray*}
\left| -\frac{L'}{L}(\sigma, \chi_4 ) \right| \leq \frac{16\cdot 3}{(1-\alpha)} \log\log |d_K| + O(1),
\end{eqnarray*} 
and again we have
\begin{eqnarray*}
\sum_{ K \in L(X),\: n_{K,C} > y} n_{K,C} =O\left(\frac X{\log X}\right).
\end{eqnarray*}

\subsubsection{Case 4. $C=(123)$} From (\ref{F_C2}),

$$\sum_{p\in C} \frac {\log p}{p^\sigma}=\frac 13\cdot \frac 1{\sigma-1}- \frac 13 \left(\frac {L'}{L}(\sigma,\chi_2)-
\frac {L'}{L}(\sigma,\chi_3)\right)+O(1).
$$
Since
$$
-\frac {L'}{L}(\sigma,\chi_2)\leq \frac 1{\sigma-1}+O(1),
$$
we have, 
$$\sum_{p\in C} \frac {\log p}{p^\sigma}\leq \frac 23\cdot \frac 1{\sigma-1} +\frac 13\cdot \frac {L'}{L}(\sigma,\chi_3)+O(1).
$$

Here $L(s,\chi_3)=\frac {\zeta_M(s)}{\zeta(s)}$, where $M$ is the cubic resolvent of $K$. Note that $M$ is an $S_3$-field. Let 
$$\tilde L(X)=\{ L(s,\chi_3)\, |\, \mbox{$M$: $S_3$-field, $|d_M| \leq X$} \}.
$$ 
By applying Kowalski-Michel's theorem to $\tilde L(X)$, we can see that every $L(s, \chi_3)$ in $\tilde L(X)$ except for $O(X^{1/1000})$ $L$-functions satisfies
\begin{eqnarray*}
\left| -\frac{L'}{L}(\sigma, \chi_3 ) \right| \leq \frac{16\cdot 3}{(1-\alpha)} \log\log |d_K| + O(1),
\end{eqnarray*} 
and with this bound, we have $n_{K,C} \ll (\log |d_K|)^{\frac{16}{(1-\alpha)A}}$, where $A=\sup_{\lambda \geq 0} \frac{1-3e^{-\lambda}}{3\lambda}=0.10...$ when $\lambda=2.29...$

Let $E(X)=\{ K \in L(X)\, |\, \mbox{$L(s,\chi_3)$ belongs to the exceptional subset in $\tilde L(X)$} \}$.
Note that there are at most $X^{\frac 12+\epsilon}$ $S_4$-fields in $L(X)$ which have the common cubic resolvent $M$. (See the Appendix: Section \ref{quartic}.) 
Hence $|E(X)| \ll X^{1/1000}  X^{\frac 12+\epsilon}$. Then
\begin{eqnarray*}
&& \sum_{ K \in L(X),\: n_{K,C} > y} n_{K,C} = \sum_{K \not\in E(X),\: n_{K,C} > y} n_{K,C} +\sum_{K \in E(X),\: n_{K,C} > y} n_{K,C} \\
&&\ll (\log X)^{\frac{16}{(1-\alpha)A}}  \sum_{K \in L(X),\: n_{K,C} > y} 1 +  X^{1/4+\epsilon} \cdot X^{1/1000}\cdot X^{\frac 12+\epsilon}=O\left(\frac X{\log X}\right).
\end{eqnarray*}

\subsubsection{Case 5. $C=e$} In \cite{CK2}, we showed that if $L(s,\chi_4)=\zeta_K(s)/\zeta(s)$ is entire and zero-free on $[\alpha,1] \times [-(\log |d_K|)^2, (\log |d_K|)^2]$, then $n_{K,e} \ll (\log |d_K|)^{\frac{16}{(1-\alpha)A}}$, where $A=\sup_{\lambda \geq 0}\frac{1-\frac 43 e^{-\lambda}}{\lambda}=0.5...$ when $\lambda=0.96...$ [There is a typo in \cite{CK2}, Theorem 1.1.] 
Since every field $K$ in $L(X)$ except for $O(X^{1/1000})$ fields has such upper bound,
$$
\sum_{K\in L(X)^{r_2},\: n_{K,C} > y} n_{K,C} =O\left(\frac X{\log X}\right).
$$

\subsection{$S_5$-fields}
We assume the strong Artin conjecture for $S_5$ fields, and follow notations in \cite{DF} for characters of $S_5$. Then $L(s,\chi_3)=\zeta_K(s)/\zeta(s)$, and $L(s,\chi_5)=\zeta_H(s)/\zeta(s)$, where $H$ is the sextic resolvent of $K$. For the sign character $\chi_2$, 
$$
\chi_4 = \chi_3 \otimes \chi_2,\quad \chi_6=\chi_5 \otimes \chi_2, \mbox{ and } \chi_7=\wedge^2 \chi_3
$$

\subsubsection{Case 1. $C=(12345)$} From (\ref{F_C2}),
$$\sum_{p\in C} \frac {\log p}{p^\sigma}=\frac 15\cdot \frac 1{\sigma-1}- \frac 15 \left(\frac {L'}{L}(\sigma,\chi_2)-\frac {L'}{L}(\sigma,\chi_3)
-\frac {L'}{L}(\sigma,\chi_4)+\frac {L'}{L}(\sigma,\chi_7)\right)+O(1).
$$

Since
\begin{eqnarray*}
\frac{L'}{L}(\sigma, \chi_3) + \frac{L'}{L}(\sigma, \chi_4) \leq 2 \sum_{ p \in C} \frac{\log p}{p^\sigma}+O(1),\quad
-\frac {L'}{L}(\sigma,\chi_2)\leq \frac 1{\sigma-1}+O(1),
\end{eqnarray*}
we have
$$\frac 35\sum_{p\in C} \frac {\log p}{p^\sigma}\leq \frac 25\cdot \frac 1{\sigma-1}- \frac 15\cdot \frac {L'}{L}(s,\chi_7)+O(1).
$$

Define $\tilde L(X)=\{L(s,\chi_7)\, |\, K \in L(X) \}.$ Note that the conductor of $L(s,\chi_7)$ is bounded by $|d_K|^7$ \cite{BK}. 
(G. Henniart noted in a private communication that it can be improved to $|d_K|^{\frac 32}$.)

\begin{Lem}\label{conj} Let $L(s,\chi_7) = L(s,\chi_7,\widehat{K}/\Q)$ and $L(s,\chi_7') = L(s,\chi_7',\widehat{K'}/\Q)$.
Suppose $L(s,\chi_7) = L(s,\chi_7')$. Then $K$ and $K'$ are conjugate. 
\end{Lem}
\begin{proof}
Recall the following from \cite{Ca}: $\chi_3=As(\sigma)$, the Asai lift of $\sigma$, which is a 2-dimensional representation of $\tilde A_5$ over $F=\Bbb Q[\sqrt{d_K}]$. Let $\pi$ be the cuspidal representation of $GL_2/F$ corresponding to $\sigma$, and let $\Pi$ be the cuspidal representation of 
$GL_4/\Bbb Q$ corresponding to $\chi_3$ by the strong Artin conjecture.
Then $\wedge^2\Pi\simeq I_F^\Q (Sym^2\pi)$.

Let $\pi', \Pi'$ be defined by $K'$. Suppose $L(s,\Pi,\wedge^2)=L(s,\Pi',\wedge^2)$. Then
$L(s,Sym^2(\pi))=L(s,Sym^2(\pi'))$. By Ramakrishnan \cite{Ra}, $\pi'\simeq \pi\otimes\chi$ for a quadratic character of $F$.
Then by Krishnamurthy \cite{Kr}, $As(\pi)\simeq As(\pi')$. Hence $\Pi\simeq \Pi'$. Therefore 
$L(s,\chi_3, K/\Bbb Q)=L(s,\chi_3, K'/\Bbb Q)$. 
Since $K$ is arithmetically solitary, $K$ and $K'$ are conjugate.
\end{proof}

Hence $L$-functions in $\tilde L(X)$ are all distinct. By applying Kowalski-Michel's theorem to $\tilde L(X)$, we can show that every 
$L(s,\chi_7)$ in $\tilde L(X)$ except for $O(X^{1/1000})$ $L$-function satisfies
\begin{eqnarray*}
\left| -\frac{L'}{L}(\sigma, \chi_7, \widehat{K}/\Q ) \right| \leq \frac{16\cdot 28}{(1-\alpha)} \log\log |d_K| + O(1).
\end{eqnarray*} 
Hence 
\begin{eqnarray*}
\sum_{K \in L(X),\: n_{K,C} > y} n_{K,C} =O\left(\frac X{\log X}\right).
\end{eqnarray*}

\subsubsection{Case 2. $C=(1234)$} From (\ref{F_C2}), 
$$\sum_{p\in C} \frac {\log p}{p^\sigma}=\frac 14\cdot \frac 1{\sigma-1}- \frac 14 \left(-\frac {L'}{L}(\sigma,\chi_2)+\frac {L'}{L}(\sigma,\chi_5)
-\frac {L'}{L}(\sigma,\chi_6)\right)+O(1).
$$
Since
$$
\frac{L'}{L}(\sigma, \chi_2) + \frac{L'}{L}(\sigma, \chi_6) \leq 2 \sum_{ p \in C} \frac{\log p}{p^\sigma}, 
$$
we have
$$
\frac{1}{2}\sum_{p \in C} \frac{\log p}{p^\sigma} \leq \frac 14\cdot \frac{1}{\sigma-1} -\frac 14\cdot \frac{L'}{L}(\sigma, \chi_5).
$$

\begin{Lem} Let $L(s,\chi_5) = L(s,\chi_5,\widehat{K}/\Q)$ and $L(s,\chi_5') = L(s,\chi_5',\widehat{K'}/\Q)$.
Suppose $L(s,\chi_5) = L(s,\chi_5')$. Then $K$ and $K'$ are conjugate. 
\end{Lem}
\begin{proof}
It is easy to see (cf. \cite{FH}, page 28) $\wedge^2 \chi_5=\chi_3\otimes\chi_2\oplus \wedge^2 \chi_3.$
Now let $\chi_3', \chi_5'$ be defined by $K'$, and suppose $\chi_5\simeq \chi_5'$. Then $\wedge^2 \chi_5\simeq \wedge^2 \chi_5'$. Hence 
$\chi_3\otimes\chi_2\oplus \wedge^2 \chi_3\simeq \chi_3'\otimes\chi_2'\oplus \wedge^2 \chi_3'$. 
By strong multiplicity one, $\wedge^2 \chi_3\simeq \wedge^2 \chi_3'$. Hence by Lemma \ref{conj}, $\chi_3\simeq \chi_3'$.
\end{proof}

Hence by applying Kowalski-Michel to $\tilde L(X)=\{ L(s,\chi_5)\, |\, K\in L(X)\}$, we proceed as in Case 1.

\subsubsection{ Case 3. $C=(12)(345)$}  From (\ref{F_C2}),
$$\sum_{p\in C} \frac {\log p}{p^\sigma}=\frac 16\cdot \frac 1{\sigma-1}- \frac 16 \left(-\frac {L'}{L}(\sigma,\chi_2)-\frac {L'}{L}(\sigma,\chi_3) + \frac {L'}{L}(\sigma,\chi_4)-\frac {L'}{L}(\sigma,\chi_5)+\frac {L'}{L}(\sigma,\chi_6)\right)+O(1).
$$
Since
$$
\frac{L'}{L}(\sigma, \chi_2) + \frac{L'}{L}(\sigma, \chi_3) + \frac {L'}{L}(\sigma,\chi_5) \leq 3 \sum_{ p \in C} \frac{\log p}{p^\sigma}, 
$$
we have
$$
\frac{1}{2}\sum_{p \in C} \frac{\log p}{p^\sigma} \leq \frac 16\cdot \frac{1}{\sigma-1} -\frac 16\cdot \frac{L'}{L}(\sigma, \chi_4)  -\frac 16\frac{L'}{L}(\sigma, \chi_6).
$$

Apply Kowalski-Michel to $\tilde L(X)=\{ L(s,\chi_4)L(s, \chi_6)\, |\, K\in L(X)\}$, and we proceed as in Case 1.

\subsubsection{Case 4. $C=(12)(34)$}  From (\ref{F_C2}), 
$$\sum_{p\in C} \frac {\log p}{p^\sigma}=\frac 18\cdot \frac 1{\sigma-1}- \frac 18\left(\frac {L'}{L}(\sigma,\chi_2)+\frac {L'}{L}(\sigma,\chi_5) + \frac {L'}{L}(\sigma,\chi_6)-2\frac {L'}{L}(\sigma,\chi_7)\right)+O(1).
$$
Since 
$$
\frac{L'}{L}(\sigma, \chi_7) \leq 2 \sum_{ p \in C} \frac{\log p}{p^\sigma}+O(1),\quad -\frac{L'}{L}(\sigma,\chi_2) \leq \frac{1}{\sigma-1}+O(1),
$$
we have
$$
\frac{1}{2}\sum_{p \in C} \frac{\log p}{p^\sigma} \leq \frac 14\cdot \frac{1}{\sigma-1} -\frac 18\left( \frac{L'}{L}(\sigma, \chi_5)  + \frac{L'}{L}(\sigma, \chi_6) \right).
$$
Apply Kowalski-Michel to $\tilde L(X)=\{ L(s,\chi_5)L(s, \chi_6)\, |\, K\in L(X)\}$, and we proceed as in Case 1.

\subsubsection{ Case 5. $C=(12)$.}
We use the $L$-function $L(s,\chi_3)$:
$$-\frac {L'}L(\sigma,\chi_3)=\sum_p \frac {\chi_3(p)\log p}{p^\sigma}+O(1).
$$
Note that $\chi_3(p)=2$ if Frob$_p\in C$, and $1+\chi_3(p)\geq 1$. Then,

$$3\sum_{p<n_{K,C}} \frac {\log p}{p^\sigma}\leq -\frac {\zeta'}{\zeta}(\sigma)-\frac {L'}L(\sigma,\chi_3)+O(1).
$$
By applying Kowalski-Michel to the set $\tilde L(X)=\{ L(s,\chi_3)\, | \, K\in L(X)\}$, we obtain that 
every $L(s, \chi_3)$ in $\tilde L(X)$ except for $O(X^{1/1000})$ $L$-function satisfies
$$
-\frac{L'}{L}(\sigma, \chi_3)\leq \frac {64}{1-\alpha}\log\log |d_K|+O(1).
$$
Hence by taking $\sigma-1=\frac {\lambda}{\log n_{K,C}}$, we have $n_{K,C}\leq (\log |d_K|)^{\frac {64}{(1-\alpha)A}}$, where $A=\sup_{ \lambda >0}\frac{2-3e^{-\lambda}}{\lambda}$.
We obtain 
$$\sum_{K \in L(X),\: n_{K,C} > y} n_{K,C} =O\left(\frac X{\log X}\right).
$$

\subsubsection{ Case 6. $C=(123)$}  
Since $\chi_3(p)=1$ if Frob$_p\in C$, we can use $L(s,\chi_3)$. This case is similar to the case $C=(12)$.

\subsubsection{ Case 7. $C=e$} 

Since $\chi_3(p)=4$ if Frob$_p\in C$, we can use $L(s,\chi_3)$. This case is similar to the case $C=(12)$.

\subsection{$S_3$-fields} For the sake of completeness, we include the case of $S_3$. Here, we follow the notations in \cite{DF} for characters of $S_3$. 

\subsubsection{Case 1. $C=(123)$}
From (\ref{F_C2}),
$$\sum_{p\in C} \frac {\log p}{p^\sigma}=\frac 13 \cdot\frac 1{\sigma-1}-\frac 13\left(\frac {L'}L(\sigma,\chi_2)-\frac {L'}L(\sigma,\chi_3)\right)+O(1).
$$
Then
$$\sum_{p<n_{K,C}} \frac {\log p}{p^\sigma}\leq \frac 23 \cdot\frac 1{\sigma-1}+\frac 13 \cdot\frac {L'}L(\sigma,\chi_3)+O(1).
$$

Since $L(s,\chi_3)=\frac {\zeta_K(s)}{\zeta(s)}$, $L(s,\chi_3)$ is modular,
i.e., $L(s,\chi_3)$ is a cuspidal automorphic $L$-function of $GL_2/\Q$. This case is similar to $S_4$, $C=(1234)$.

\subsubsection{Case 2. $C=(12)$}

From (\ref{F_C2}),

$$\sum_{p\in C} \frac {\log p}{p^\sigma}=\frac 12 \cdot\frac 1{\sigma-1}+\frac 12\cdot\frac {L'}L(\sigma,\chi_2)+O(1).
$$
This case was done in Section \ref{CB}.

\subsubsection{Case 3. $C=e$}

This case is similar to $S_4$, $C=e$.

We summarize it as

\begin{Thm} Theorem \ref{main} holds unconditionally for $S_3$ and $S_4$-fields. Under the strong Artin conjecture for $S_5$-fields, Theorem \ref{main} holds.
\end{Thm}

\begin{Rmk} The above method can be generalized to $S_n$-fields and a special conjugacy class: Let $K$ be an $S_n$-field, and let
$L(s,\chi)=\frac {\zeta_K(s)}{\zeta(s)}$. Let $C$ be a conjugacy class of $S_n$ such that $\chi(C)\geq 1$. Then under the counting conjectures $(\ref{estimate})-(\ref{estimate1})$ and the strong Artin conjecture for $L(s,\chi)$, we have
\begin{eqnarray*}
\frac{1}{|L_n^{(r_2)}(X)|} \sum_{ K \in L_n^{(r_2)}(X)} n_{K,C} = \sum_{ q } \frac{q(1-|C|/|S_n|+f(q)) }{1+f(q)}\prod_{p<q} \frac{|C|/|S_n|}{1+f(p)}+O\left( \frac{1}{\log X}\right).
\end{eqnarray*}
\end{Rmk}

\section{Average value of $N_{K,C}$}\label{N}

In this section, we compute the averages of $N_{K,C}$. We need to assume the GRH for Dedekind zeta functions because we do not have a good bound on $N_{K,C}$. We generalize Martin and Pollack's result for $S_3$-fields. (See Theorem 4.8 \cite{MP}.) Since the idea of proof is similar to the case of $n_{K,C}$, we omit some details. Consider
\begin{eqnarray*}
\sum_{K \in L(X)} N_{K,C} = \sum_{K \in L(X),\: N_{K,C} \leq y } N_{K,C} + \sum_{K \in L(X),\: N_{K,C} > y} N_{K,C} .
\end{eqnarray*}

Here $N_{K,C}=q$ means that for all primes $p < q$, Frob$_p \notin C$ and Frob$_q \in C$. By the counting conjectures, there are
$$
\frac{|C|/|S_n|}{1+f(q)}\prod_{p<q} 
\frac{1-|C|/|S_n|+f(p) }{1+f(p)} A(r_2)X + O(X^{\frac{1+3\delta}{4}}).
$$
such number fields in $L(X)$. Hence,
\begin{eqnarray*}
\sum_{K \in L(X),\: N_{K,C} \leq y } N_{K,C}& =& \sum_{ q \leq y} q \sum_{K \in L(X),\: N_{K,C}=q} 1 \\
  &=& A(r_2)X \sum_{q \leq y} \frac{q(|C|/|S_n|)}{1+f(q)}\prod_{p<q} 
\frac{1-|C|/|S_n|+f(p) }{1+f(p)} + O(y^2X^{\frac{1+3\delta}{4}})\\
 &=& A(r_2)X \sum_{ q } \frac{q(|C|/|S_n|)}{1+f(q)}\prod_{p<q} 
\frac{1-|C|/|S_n|+f(p) }{1+f(p)}+ O\left( \frac{X}{\log X}\right).
\end{eqnarray*}

In order to estimate the second sum $\sum_{K \in L(X),\: N_{K,C} > y} N_{K,C}$, we need GRH, which implies that $N_{K,C}\ll (\log |d_K|)^2$ (cf. \cite{BaS}). Then
\begin{eqnarray*}
\sum_{K \in L(X),\: N_{K,C} > y} N_{K,C}\ll (\log X)^2
\sum_{N_{K,C} > y} 1 
&&\ll X(\log X)^2 \prod_{p<y} \frac {1-|C|/|S_n|+f(p)}{1+f(p)} \\
&&\ll X(\log X)^2 \left(\frac {|C|}{|S_n|}\right)^{\pi(y)} =O\left(\frac X{\log X}\right).
\end{eqnarray*}
Hence
\begin{Thm} \label{main-N}
Let $L_n^{(r_2)}(X)$ be the set of $S_n$-fields $K$ of signature $(r_1,r_2)$ with $|d_K|<X$. Assume the counting conjectures $(\ref{estimate})-(\ref{estimate1})$ and the GRH for Dedekind zeta functions. Let $C$ be a conjugacy class of $S_n$ and $N_{K,C}$ be the least prime with Frob$_p \in C$. Then,  
\begin{eqnarray*}
\frac{1}{|L_n^{(r_2)}(X)|} \sum_{ K \in L_n^{(r_2)}(X)} N_{K,C} = \sum_{ q } \frac{q(|C|/|S_n|)}{1+f(q)}\prod_{p<q} \frac{1-|C|/|S_n|+f(p)}{1+f(p)}+O \left( \frac{1}{\log X}\right).
\end{eqnarray*}
\end{Thm}

The tables below show average values of $N_{K,C}$ for $S_3$, $S_4$, and $S_5$-fields. The computations are done by PARI. The average values of $N_{K,C}$ for $S_3$ are given in \cite{MP}. 
\begin{center}
\begin{tabular}{|c|c|}
  \hline
  $S_3$ & \mbox{Average of $N_{K,C}$} \\
\hline
$[e]$ & $19.79522...$ \\
\hline
$[(12)]$ & $5.36802...$\\
\hline
$[(123)]$ & $8.54472...$\\
\hline 
\end{tabular}
\quad
\begin{tabular}{|c|c|}
  \hline
  $S_4$ & \mbox{Average of $N_{K,C}$} \\
\hline
$[e]$ & $108.71075...$ \\
\hline
$[(12)(34)]$ & $28.96178...$\\
\hline
$[(1234)]$ & $12.69279...$\\
\hline 
$[(12)]$ & $12.69279...$\\
\hline 
$[(123)]$ & $9.098479...$\\
\hline 
\end{tabular}
\quad
\begin{tabular}{|c|c|}
  \hline
  $S_5$ & \mbox{Average of $N_{K,C}$} \\
\hline
$[e]$ & $716.34521...$ \\
\hline
$[(12)(34)]$ & $29.19651...$\\
\hline
$[(123)]$ & $20.75158...$\\
\hline 
$[(12)(345)]$ & $20.75158...$\\
\hline 
$[(12)]$ & $47.44681...$\\
\hline 
$[(1234)]$ & $12.88664...$\\
\hline 
$[(12345)]$ & $16.72312...$\\
\hline 
\end{tabular}
\end{center}

\section{Unconditional results on $N_{K,C}$}\label{N-b}
We can't remove the GRH hypothesis from Theorem \ref{main-N} for a single conjugacy class. However, we may have an unconditional result for the union of conjugacy classes not contained in $A_n$. Let $C'$ be the union of all the conjugacy classes not contained in $A_n$ and 
$N_{K,C'}$ be the smallest prime for which $Frob_p \in C'$. 

\begin{Thm} \label{N-uncond}
Assume the counting conjectures $(\ref{estimate})-(\ref{estimate1})$, and Conjecture \ref{QRconj2} with $\beta_n>\frac 1{4\sqrt{e}}$. Then 
\begin{eqnarray*}
\frac{1}{|L_n^{(r_2)}(X)|} \sum_{ K \in L_n^{(r_2)}(X)} N_{K,C'} = \sum_{ q } \frac{\frac 12 q}{1+f(q)}\prod_{p<q} \frac{ \frac 12+f(p)}{1+f(p)}+O \left( \frac{1}{\log X}\right).
\end{eqnarray*}
\end{Thm}

Since the proof is similar to previous ones, we just note unconditional and conditional bounds on $N_{K,C'}$. Since $d_K=d_F m^2$ for some integer $m$, if $Frob_p \in C'$, then $p$ is inert in $F$. (i.e., $\left(\frac{d_F}{p}\right)=\left( \frac{d_K}{p} \right)=-1$.) Conversely, if $p$ is inert in $F$ and $p\nmid d_K$ (i.e., $\left( \frac{d_K}{p}\right)=-1$), then $Frob_p$ is in $C'$. Hence, $N_{K,C'}$ is the smallest prime such that $\left( \frac{d_K}{p}\right)=-1$. Hence, by Norton \cite{No},
$$
N_{K,C'} \ll |d_K|^{\frac{1}{4\sqrt{e}}+\epsilon}.
$$ 
(Norton's result is valid for imprimitive characters.) 

Now, we obtain a conditional bound on $N_{K,C'}$.
\begin{Prop} \label{CBN_C}
Let $F$ be a quadratic number field $\Q(\sqrt{d_F})$. Assume that $L(s,\chi_F)$ is zero-free in $[1-\alpha] \times [-(\log |d_F|)^2, (\log |d_F|)^2]$. Then there is a positive constant $A$ independent of $\alpha$ for which there exists a prime $p$ which is inert and $p\ll (A\alpha \log |d_F|)^{1/\alpha}$. 
\end{Prop}
\begin{proof} The proof is essentially the same as that of Proposition 4.2 in \cite{CK1}.
\end{proof}

If $K$ has a quadratic resolvent $F$ for which $L(s,\chi_F)$ has the desired zero-free region, then by Proposition \ref{CBN_C},
$N_{K,C'} \ll ( A \alpha \log |d_K|)^{1/\alpha}.$ 
Then by using Conjecture \ref{QRconj2}, Theorem \ref{N-uncond} follows.
Since the assumptions in Theorem \ref{N-uncond} hold for $S_3$-fields, we have

\begin{Cor} Theorem \ref{N-uncond} holds unconditionally for $S_3$-fields.
For $S_3$-fields and $C=[(12)]$, the average value of $N_{K,C}$ is $5.36802\dots$.
\end{Cor}

This was proved in Martin and Pollack under GRH (\cite{MP}, Theorem 4.8).

\begin{Cor} Under Conjecture \ref{QRconj2}, Theorem \ref{N-uncond} holds for $S_4$ and $S_5$-fields.  For $S_4$-fields and $C'=[(1234)]\cup [(12)]$, the average value of $N_{K,C'}$ is $5.821569\dots$.
For $S_5$-fields and $C'=[(12)(345)] \cup [(12)] \cup [(1234)]$, the average value of $N_{K,C'}$ is $5.9733589\dots$. 
\end{Cor}

\section{Appendix: $S_4$-fields with the same cubic resolvent}\label{quartic}

Given a noncyclic cubic field $M$, let
$$\Phi_M(s)=1+\sum_{K\in \mathcal F(M)} \frac 1{f(K)^s},
$$
where $d_K=d_M f(K)^2$, and $\mathcal F(M)$ is the set of all $S_4$-fields $K$ with the cubic resolvent field $M$. 
Let $\mathcal L(M,n^2)$ be the set of quartic fields whose cubic resolvents are isomorphic to $M$ and whose discriminants are $n^2 d_M$, and 
$\mathcal L_{tr}(M,64)$  the subset of $\mathcal L(M,64)$, where 2 is totally ramified.  Define $\mathcal L_2(M)=\mathcal L(M,1)\cup \mathcal L(M,4)\cup \mathcal L(M,16)\cup \mathcal L_{tr}(M,64)$.  
By Kl\"{u}ners \cite{K}, 
 $|\mathcal L(M,n)|\ll (n^2 |d_M|)^{\frac 12+\epsilon}$, hence $|\mathcal L_2(M)| \ll  |d_M|^{\frac 12 +\epsilon}$. 

By Theorem 1.4 in Cohen and Thorne \cite{CT1},
$$
\Phi_M(s)=\sum_{i=1}^{|\mathcal L_2(M)+1|}\Phi_i(s),\quad \Phi_i(s)=\sum_{n=1}^\infty \frac{a_i(n)}{n^s}, 
$$
and $a_i(n) \leq 3^{\omega(n)} \ll 3^{\frac{\log n}{\log \log n}} \ll n^{\epsilon}$. Also Theorem 1.4 in \cite{CT1} implies 
$$
\Phi_i(1+c+ it) \ll \left( \frac{\zeta(1+c)}{\zeta(2+2c)}\right)^3 \ll \frac{1}{c^3}. 
$$

By applying Perron's formula to each $\Phi_i(s)$ for $i=1,2,\dots, |\mathcal L_2(M)+1|$, we can obtain that
$$|\{K\in \mathcal F(M)\, |\, f(K)\leq x\}|
\ll x(\log x)^4 |d_M|^{\frac 12+\epsilon}.
$$

Hence we have proved
\begin{Prop}
Let $CR_4(X,M)$ be the set of $S_4$-fields $K$ with the given cubic resolvent $M$ and $|d_K|\leq X$. Then
$$|CR_4(X,M)|\ll X^{\frac 12} (\log X)^4 |d_M|^{\epsilon},
$$
with an absolute implied constant.
\end{Prop}

\end{document}